\documentclass[a4paper]{amsart}

\usepackage{amsmath,amssymb,amsfonts,amsthm,bm}
\usepackage{mathrsfs}

\theoremstyle{plain}
\newtheorem{theorem}{Theorem}
\newtheorem{proposition}[theorem]{Proposition}
\newtheorem{lemma}[theorem]{Lemma}
\newtheorem{corollary}[theorem]{Corollary}

\theoremstyle{definition}
\newtheorem{example}[theorem]{Example}
\newtheorem{definition}[theorem]{Definition}

\theoremstyle{remark}
\newtheorem{remark}[theorem]{Remark}

\numberwithin{theorem}{section} 
\numberwithin{equation}{section}

\newtheorem{mainthm}{Theorem} 
\newtheorem{maindfn}[mainthm]{Definition}

\newcommand{\abs}[1]{\left\lvert #1 \right\rvert}
\newcommand{\norm}[1]{\left\lVert #1 \right\rVert}

\begin{document}

\title[Singular value functions for C\(^*\)-algebras]{Singular value functions for C\(^*\)-algebras}
\author{Naoto Fujitsu}
\address{Department of Pure and Applied Mathematics, Graduate School of Information Science 
and Technology, The University of Osaka, Yamadaoka 1-5, Suita, Osaka 565-0871, Japan}
\email{fujitsu-n@ist.osaka-u.ac.jp}
\keywords{singular value function, singular value, ordered K-theory}
\subjclass[2020]{Primary 46L51, Secondary 46L05;47B06}

\begin{abstract}
  We introduce \textit{singular value functions} for C\(^*\)-algebras, generalizing the singular values of compact operators on Hilbert spaces.
  We also establish several fundamental properties of these singular value functions and present examples of singular value functions for certain classes of C\(^*\)-algebras. 
\end{abstract}

\maketitle

\section{Introduction}

In this article, 
we define singular value functions for C\(^*\)-algebras, 
motivated by the theory of singular value functions for von Neumann algebras.

The notion of singular value functions for von Neumann algebras was introduced as a generalization of the singular value sequence for \(\mathcal{K}\), the algebra of compact operators on a Hilbert space (see \cite[pp.~560--562]{Dodds.Pagter.Sukochev:2023:Noncommutative_Integration_and_Operator_Theory}). 
The singular value functions for II\(_1\)-factors were introduced by Murray and von Neumann in \cite{Murray.von_Neumann:I}.
The singular value functions for semifinite von Neumann algebras with faithful traces are formulated, for example, in \cite{TFack:1982:SVF, Ovchinnikov:70:s-numbers, Sonis:71}, 
and their fundamental properties are clarified in \cite{TFack.Kosaki:1986:Generalized_s-numbers}.
Let \(\mathcal{M}\) be a semifinite von Neumann algebra with a faithful trace \(\tau\).
In this case, singular value functions are not necessarily sequences indexed by nonnegative integers, but functions on \(\mathbb{R}^+\) (more precisely, on \(\tau(P(\mathcal{M})) \subset \mathbb{R}^+\)).
Indeed, for \(a \in \mathcal{M}\) and \(t \in \mathbb{R}^+\), the singular value function is defined by
\begin{equation*}
  s_t(a):=\inf\{\norm{a-ap}\mid p \text{ is a projection in } \mathcal{M} \text{ with } \tau(p)\leq t\}.
\end{equation*}
Accordingly, the singular value function of an element \(a\in\mathcal{M}\) is the map
\begin{equation*}
  s(a): [0,\infty)\longrightarrow\mathbb{R}^+,\qquad t\longmapsto s_t(a).
\end{equation*}

Motivated by the above construction, we introduce singular value functions for C\(^*\)-algebras.
It is natural to define singular value functions for a C\(^*\)-algebra on equivalence classes of projections.
Accordingly, we take the positive cone of the \(K_0\)-group as their domain.

\begin{maindfn}[Definition~\ref{SVF:definition}]
  Let \(A\) be a C\(^*\)-algebra. 
  For \(a\in A\), 
  we define a function \(s(a)\) on \(K_0(A)^+\) by  
  \begin{equation*}
    s_g(a):=\inf\{\norm{a-ap}\mid{p}\in{P(A)},\;[p]_0\leq g\}\in\mathbb{R}^+,
  \end{equation*}
  for each \(g\in K_0(A)^+\). 
  We call the function \(s(a)\) the \textit{singular value function} of \(a\). 
  We also regard \(s\) as a map
  \begin{equation*}
    s \colon A \longrightarrow \{\, f \mid f \colon K_0(A)^+ \to \mathbb{R}^+ \,\}, \qquad a \longmapsto s(a).
  \end{equation*}
\end{maindfn}

By definition, the singular value function \(s(a)\) is nonnegative and decreasing for any element \(a\) in a C\(^*\)-algebra \(A\).
These functions coincide with the classical singular values on \(\mathcal{K}\).
In general, we show that 
the singular value functions of positive elements with finite spectrum 
are step functions 
with respect to the order structure of the positive cone of the \(K_0\)-group of \(A\).

We investigate several basic properties of the singular value functions for C\(^*\)-algebras. 
The singular value functions for semifinite von Neumann algebras satisfy properties analogous to those of singular values for \(\mathcal{K}\).
It is therefore natural to ask whether singular value functions for C\(^*\)-algebras enjoy analogous properties.
We establish the following results.

\begin{mainthm}
  Let \(A\) be a C\(^*\)-algebra with cancellation and real rank zero.
  If \(K_0(A)^+\) coincides with the dimension range of \(A\),
  then for \(a,b\in A\), the following hold:
  \begin{enumerate}
    \item \(s_0(a)=\norm{a}\).
    \item \(s(\alpha a)=\abs{\alpha}s(a)\) for \(\alpha\in\mathbb{C}\).
    \item \(s(ba)\leq \norm{b}s(a),\;s(ab)\leq \norm{b}s(a)\).
    \item \(\abs{s_g(a)-s_g(b)}\leq\norm{a-b}\) for \(g\in K_0(A)^+\).
    \item \(s(a)=s(\abs{a})\).
    \item \(s(a)=s(a^*)\).
    \item \(s_{g+h}(a+b)\leq s_g(a)+s_h(b)\) for \(g,h\in K_0(A)^+\).
    \item \(s_{g+h}(ab)\leq s_g(a)s_h(b)\) for \( g,h\in K_0(A)^+\).
    \item Let \(f\colon \mathbb{R}^+\to\mathbb{R}^+\) be a continuous increasing function with \(f(0)=0\). 
    If \(a\in A^+\), then \(s(f(a))=f(s(a))\).
    \item If \(0\leq a\leq b\), then \(s(a)\leq s(b)\).
  \end{enumerate}
\end{mainthm}

The inequalities in (7) and (8) are analogues of Ky Fan's inequalities \cite{KyFan:1951:inequalities_for_eigenvalues}.
Some of these properties hold under weaker assumptions on the C\(^*\)-algebra.
See 
Proposition~\ref{SVF:Basic_Property:1}, 
Corollary~\ref{SVF:Basic_Property:2},
Theorem~\ref{SVF:Basic_Property:3},
Theorem~\ref{SVF:Basic_Property:4},
Proposition~\ref{SVF:Basic_Property:5}
and Proposition~\ref{SVF:Basic_Property:6}
for the proofs.

While the equivalence classes of projections in \(\mathcal{K}\) and semifinite von Neumann algebras form a lattice as an ordered set, 
the positive cone of the \(K_0\)-group of a C\(^*\)-algebra does not necessarily have a lattice structure. 
This causes essential difficulties in the proof. 
To overcome this issue, we employ approximation arguments and the theory of comparison of positive elements as in \cite{Cuntz:1978:Dimension_functions_on_simple_C*-algebras,Kirchberg.Rordam:2000:non-simple.purely.infinite.C*-algebras}.

We also study the right-continuity of singular value functions for C\(^*\)-algebras.
In the case of semifinite von Neumann algebras with faithful traces, 
singular value functions are right-continuous decreasing functions on \(\mathbb{R}^+\).
In contrast, 
since \(K_0(A)^+\) is not a priori equipped with a topology, 
it is not always possible to define a notion of ``right-continuity'' for functions on \(K_0(A)^+\).
On the other hand, if \(K_0(A)^+\) is totally ordered and Archimedean, it can be identified with a subset of \(\mathbb{R}^+\), and hence endowed with a natural topology.
Under these assumptions, we obtain the following theorem.

\begin{mainthm}[Theorem~\ref{SVF:form:simple_and_totally_ordered}]
  Let \(A\) be a C\(^*\)-algebra with cancellation such that the ordered \(K_0\)-group of \(A\) is simple and totally ordered.
  If \(A\) is unital (respectively stable), then the following hold:
  \begin{enumerate}
    \item Let \(f\colon K_0(A)^+ \to \mathbb{R}^+\) be a decreasing, right-continuous function such that \(DP(\overline{f})\subset K_0(A)^+\) and \(f([1_A]_0)=0\) (respectively vanishes at infinity). 
    There exists \(a \in A\) such that \(f = s(a)\).
    \item If \(A\) has real rank zero, 
    then for every \(a\in A\), 
    the singular value function \(s(a)\) is right-continuous, vanishes at infinity, and \(DP(\overline{s(a)})\subset K_0(A)^+\).
  \end{enumerate}
\end{mainthm}

We show that 
if \(K_0(A)\) is not totally ordered, then
singular value functions are not necessarily lower semicontinuous 
with respect to a product topology on \(K_0(A)\) induced by 
the discrete topology on the subgroup of infinitesimal
and the topology of \(\mathrm{Aff}(S(K_0(A)))\).

This paper is organized as follows. 
In Section~\ref{section:Preliminaries}, 
we collect notation and several preliminary lemmas. 
In Section~\ref{section:SVF}, 
we introduce the definition of the singular value functions for C\(^*\)-algebras and establish their basic properties.
In Section~\ref{section:realization-continuity}, 
we show that certain types of functions on the positive cone of a \(K_0\)-group are realized by singular value functions. 

\section{Preliminaries}

\label{section:Preliminaries}

In this section we introduce notation and present some lemmas.

\subsection{Notation}

In this paper, 
\(\mathbb{N}\) contains \(0\) 
and \(\mathbb{R}^+:=[0,\infty)\).

Let \(a\) and \(b\) be elements in a C\(^*\)-algebra \(A\). 
We write \(a\oplus b\) for the diagonal matrix \(\operatorname{diag}(a,b)\in M_2(A)\).
The unitization of \(A\) is denoted by \(\tilde{A}\). 
If \(A\) is unital, then \(\tilde{A}=A\), unless otherwise specified. 
For a unital C\(^*\)-algebra \(A\),
we write \(1_A\) for the unit of \(A\), 
\(U(A)\) for the unitary group of \(A\), 
and \(A^{-1}\) for the set of invertible elements of \(A\).
We denote by \(M(A)\) the multiplier algebra of \(A\).

We denote by \(A^+\) the set of positive elements in a C\(^*\)-algebra \(A\).
For \(a\in A^+\) and \(\varepsilon>0\), we denote by \((a-\varepsilon)_+\) the positive part of the self-adjoint element \(a-\varepsilon1_{\tilde{A}}\) in \(\tilde{A}\).
We denote by \(\overline{aAa}\) the hereditary sub-C\(^*\)-algebra of \(A\) generated by \(a\).
For positive elements \(a,b\) in \(A\), 
\(a\sim_{\mathrm{MvN}}b\) denotes that \(a\) and \(b\) are Murray--von Neumann equivalent.
We say \(a\) is \textit{Cuntz sub-equivalent} to \(b\) if 
there exists a sequence \(\{x_k\}_{k=1}^\infty\) in \(A\) such that \(x_k^*bx_k\) converges to \(a\), writing \(a\lesssim b\).
The notion of Cuntz sub-equivalence was introduced by Joachim Cuntz in \cite{Cuntz:1978:Dimension_functions_on_simple_C*-algebras}.
If \(a\sim_{\mathrm{MvN}}b\), then clearly \(a\lesssim b\) and \(b\lesssim a\).

We denote by \(P(A)\) the set of projections in \(A\). 
For projections \(p,q\) in \(P(A)\), 
we say that \(p\) is a \textit{subprojection} of \(q\) if \(p\leq q\), that is, \(pq=qp=p\).
Moreover, \(p\) is said to be \textit{subordinate} to \(q\), or \textit{majored} by \(q\), 
if there exists a projection \(p'\in P(A)\) such that \(p'\sim_{\mathrm{MvN}}p\) and \(p'\leq q\). 
It is well known that \(p\) is Cuntz sub-equivalent to \(q\) if and only if \(p\) is subordinate to \(q\). 
Thus we use the notation \(\lesssim\) for both relations.
The dimension range of \(A\) is denoted by 
\begin{equation*}
  D_0(A):=\{[p]_0\mid p\in P(A)\}.
\end{equation*}
We say that \(A\) has the cancellation property, or simply has cancellation, 
if for every \(p,q\in P_\infty(A):=\bigcup_{n\in\mathbb{N}_{\geq 1}} P(M_n(A))\),
\begin{equation*}
  [p]_0=[q]_0\iff p\sim_{\mathrm{MvN}}q.
\end{equation*}

Let \(a\) be a positive element of a C\(^*\)-algebra \(A\) with finite spectrum.
By functional calculus, 
there exist a positive integer \(n\), mutually orthogonal nonzero projections
\(p_1,p_2,\dots,p_n \in A\) 
and 
positive real numbers
\(\alpha_0\geq\alpha_1\geq\cdots\geq\alpha_{n-1}\)
such that 
\begin{equation*}
  a=\sum_{i=1}^{n}\alpha_{i-1}p_i.
\end{equation*}

For mutually orthogonal nonzero projections \(p_1,p_2,\dots,p_n\) in \(A\),
we define an increasing sequence of projections \(\{\hat{p}_k\}_{k=0}^{n}\) by
\begin{equation*}
  \hat{p}_k:=\sum_{i=0}^{k}p_i,
\end{equation*}
where \(p_0\) is the zero projection. 

We denote by \(\mathcal{K}\) the C\(^*\)-algebra of compact operators on an infinite-dimensional separable Hilbert space.
We denote by \(T(A)\) the set of tracial states on \(A\).

\begin{lemma}
  \label{subordinate:SC}
  Let \(p\) and \(q\) be projections in a C\(^*\)-algebra \(A\).
  If \(\norm{p-pq}<1\), then \(p\lesssim q\).
\end{lemma}

\begin{proof}
  This lemma is well-known, but we give a proof for the reader's convenience.
  Since \(\norm{p-pqp}\leq\norm{p-pq}<1\),  
  it follows from the Neumann series
  that \(pqp\) is an invertible element in the unital C\(^*\)-subalgebra \(pAp\). 
  Let \(b\) be the inverse of \(pqp\) in \(pAp\). 
  As \(pqp\) is positive, so is \(b\).  
  Set \(u^*:=b^{1/2}pq\). 
  Since \(uu^*=qpb pq\in P(A)\), 
  we compute
  \begin{align*}
    u^*u&=b^{1/2}pqq pb^{1/2}=p,\\
    uu^*&=qpb pq\leq q\norm{qpbpq}q=q.
  \end{align*}
  Hence we have \(p \lesssim q\).
\end{proof}

\begin{lemma}
  \label{Cuntz-subeq2subproj}
  Let \(A\) be a C\(^*\)-algebra with cancellation. 
  \begin{enumerate}
    \item If \(p_1, q, p_2 \in P(A)\) satisfy \(p_1 \leq p_2\) and 
    \(p_1 \lesssim q \lesssim p_2\) in \(A\), 
    then there exists a projection \(q'\) in \(A\) with \(q'\sim_{\mathrm{MvN}}q\) and \(p_1\leq q'\leq p_2\).
    \item Let \(\{p_n\}_{n=1}^{N}\) and \(\{q_m\}_{m=1}^{M}\) be increasing sequences in \(P(A)\).
    If there exists \(1\leq m_1<\cdots< m_N=M\) with \(p_n\sim_{\mathrm{MvN}}q_{m_n}\) for all \(1\leq n\leq N\), 
    then there exists an increasing sequence \(\{r_m\}_{m=1}^{M}\) in \(P(A)\)
    such that 
    \begin{equation*}
      r_m\sim_{\mathrm{MvN}} q_m\quad\text{for all \(1\leq m\leq M\)},
    \end{equation*}
    and
    \begin{equation*}
      r_{m_n}=p_n\quad\text{for all \(1\leq n\leq N\)}.
    \end{equation*} 
  \end{enumerate}
\end{lemma}

\begin{proof}
  (1)
  There exists \(q''\) in \(P(A)\) with \(q\sim_{\mathrm{MvN}}q''\leq p_2\),
  and there exists \(p'\in P(A)\) such that \(p_1\sim_{\mathrm{MvN}}p'\leq q''\leq p_2\).
  By \([p_1]_0=[p']_0\) and \(p_1,p'\leq p_2\), we have
  \begin{equation*}
    [p_2]_0=[(p_2-p_1)\oplus p_1]_0=[(p_2-p')\oplus p']_0.
  \end{equation*}
  Therefore, \([p_2-p_1]_0=[p_2-p']_0\).
  Since \(A\) has cancellation,
  it follows that \(p_2-p_1\sim_{\mathrm{MvN}} p_2-p'\).
  By \cite[Proposition 2.2.2]{Rordam.Larsen.Laustsen},
  there exists a unitary element \(u\) in a unital sub-C\(^*\)-algebra \(p_2Ap_2\) of \(A\) 
  such that \(p_1=u^*p'u\). 
  We define a projection \(q':=u^*q''u\) in \(A\). 
  Since \(p'\leq q''\), we have \(p_1\leq q'\leq p_2\) and clearly \(q'\sim_{\mathrm{MvN}}q\).

  (2) 
  Put \(m_0:=0,\;p_0:=0\) and \(r_0:=0\).
  We construct \(r_i\) by induction on \(i\).

  Suppose that increasing sequence \(\{r_m\}_{m=0}^{i-1}\) is constructed and 
  \begin{equation*}
    r_j\leq p_n\quad\text{for all \(1\leq j \leq m_n,\;1\leq n\leq N\)}.
  \end{equation*}
  Let \(n\) be the unique integer such that \(m_{n-1}<i\leq m_n\).
  Assume inductively that
  \begin{equation*}
    p_{n-1}\leq r_{i-1}\leq p_n\quad\text{if \(m_{n-1}<i\leq m_n\)}.
  \end{equation*} 
  If \(i=m_n\), set \(r_i:=p_n\).
  Otherwise, 
  applying (1) to \(r_{i-1}\lesssim q_i\lesssim p_n\) and \(r_{i-1}\leq p_n\), 
  there exists \(r\in P(A)\) with \(r\sim_{\mathrm{MvN}}q_i\) and \(r_{i-1}\leq r\leq p_n\).
  We set \(r_i:= r\).
  Thus we obtain \(r_i\) with \(r_i\sim_{\mathrm{MvN}}q_i\) and \(r_i\leq p_n\).
  This completes the induction.
\end{proof}

\begin{lemma}
  \label{projection:depart:sequence}
  Let \(A\) be a C\(^*\)-algebra with cancellation.
  \begin{enumerate}
    \item Let \(p\) be a projection in \(A\) and let \(\{g_n\}_{n=1}^{N}\) be an increasing sequence in \(K_0(A)^+\). 
    If \(g_N=[p]_0\), 
    then there exists an increasing sequence \(\{p_n\}_{n=1}^{N}\) in \(P(A)\) such that \([p_n]_0=g_n\) for \(1\leq n\leq N\) and \(p_N=p\).
    \item Let \(\{p_n\}_{n=1}^{N}\) be an increasing sequence in \(P(A)\) 
    and \(\{g_m\}_{m=1}^{M}\) be an increasing sequence in \(K_0(A)^+\)
    such that there exists an increasing sequence \(\{m_n\}_{n=1}^{N}\) with
    \begin{equation*}
      g_{m_n}=[p_n]_0 \text{ for all \(1\leq n\leq N\)}\quad\text{and}\quad m_N=M.
    \end{equation*}
    There exists an increasing sequence \(\{q_m\}_{m=1}^{M}\) in \(P(A)\) such that 
    \begin{equation*}
      [q_m]_0=g_m \text{ for all \(1\leq m\leq M\)},
    \end{equation*}
    and 
    \begin{equation*}
      p_{n}=q_{m_n} \text{ for all \(1\leq n\leq N\)}.
    \end{equation*}
  \end{enumerate}
\end{lemma}

\begin{proof}
  (1)
  By applying \cite[Lemma~3.2~{(1)}]{Rordam:2004:Stable_C*-algebras} repeatedly, 
  we obtain an increasing sequence with the required properties.

  \medbreak

  (2)
  Applying (1) to \(\{g_m\}_{m=1}^{M-1}\) and \(g_M=[p_N]_0\), 
  we obtain an increasing sequence \(\{q'_m\}\) of projections with \([q'_m]_0=g_m\) for all \(1\leq m\leq M\) and \([q'_{m_n}]_0=[p_n]_0\) for all \(1\leq n\leq N\).
  Since \(A\) has cancellation, 
  \begin{equation*}
    q'_{m_n}\sim_{\mathrm{MvN}}p_n\text{ for all \(1\leq n\leq N\)}.
  \end{equation*}
  Applying Lemma~\ref{Cuntz-subeq2subproj}~(2) to \(\{p_n\}\) and \(\{q'_m\}\),
  we obtain the required sequence.
\end{proof}

\subsection{Step function}

Let \(S\) be a subset of \(\mathbb{R}^+\) containing \(0\).
A function \(f\) on \(S\) is said to be \textit{right-continuous}
if \(f\) is right-continuous with respect to the relative topology on \(S\).

Let \(F\) be a finite subset of \(S\) containing \(0\).
Write the elements of \(F\) as
\begin{equation*}
  0=x_0<x_1<\cdots<x_{\abs{F}-1}.
\end{equation*}
For a decreasing function \(f\colon S\to\mathbb{R}^+\), 
we define a decreasing right-continuous step function \(g^{f}_{F}\) on \(S\) by
\begin{equation*}
  g_{F}^{f}:=\sum_{i=1}^{\abs{F}-1}f(x_{i-1})\chi_{[x_{i-1},x_{i})\cap S}.
\end{equation*}
We set \(g_{F}^{f}=0\) on \([x_{\abs{F}-1},\infty)\cap S\).
The sup norm of the difference can be written as follows:
\begin{equation*}
  \begin{split}
    \norm{g_{F}^{f}-f}
    &=
    \norm{g_{F}^{f}-f}_{S}
    =
    \sup\left\{\abs{(g_{F}^{f}-f)(y)}\mid y\in S\right\}\\
    &=
    \max
    \left(
      \left\{
        \sup_{y\in[x_{i-1},x_{i})\cap S}(f(x_{i-1})-f(y))
      \right\}_{1\leq i\leq \abs{F}-1}
      \cup
      \left\{
        f(x_{\abs{F}-1})
      \right\}
    \right).
  \end{split}
\end{equation*}

Suppose that \(f\colon S\to\mathbb{R}^+\) is a right-continuous and decreasing function, 
and let \(\overline{S}\) be the closure of \(S\).
For \(x\in \overline{S}\), we set
\begin{align*}
  f(x+)&:=
  \sup\{f(y)\mid y\in (x,\infty)\cap S\}, \text{ whenever \((x,\infty)\cap S\neq \emptyset\)},\\
  f(x-)&:=
  \begin{cases}
    \inf\{f(y)\mid y\in [0,x)\cap S\}, & \text{if \(x\neq 0\)},\\
    f(0), & \text{if \(x=0\)}.
  \end{cases}
\end{align*}
If \(x\) is a right, respectively left, accumulation point of \(S\),
then these values coincide with the corresponding topological one-sided limits.
We define a function \(\overline{f}\colon\overline{S}\to\mathbb{R}^+\) by 
\begin{equation*}
  \overline{f}(x):=
  \begin{cases}
    f(x) & \text{if \(x\in S\)},\\
    f(x+) & \text{if \(x\in\overline{S}\setminus S\) is a right-accumulation point},\\
    f(x-) & \text{otherwise}.
  \end{cases}
\end{equation*}
This function \(\overline{f}\) is right-continuous and decreasing.
We denote by
\begin{equation*}
  DP(\overline{f}):=\left\{x\in \overline{S}\setminus\{0\}\mid \overline{f}(x-)>\overline{f}(x)\right\},
\end{equation*}
the set of left-jump points of \(\overline{f}\).

\begin{lemma}
  Let \(S\) be a subset of \(\mathbb{R}^+\) with \(0\in S\),
  and let \(f\colon S\to\mathbb{R}^+\) be a right-continuous decreasing function with \(DP(\overline{f})\subset S\).

  Let \(a,b\in S\) with \(a<b\).
  Then the following statements hold.
  \begin{enumerate}
    \item  If \(0\leq \lambda < f(a)-f(b-)\),
    then for any \(\mu>0\), there exists \(c\in [a,b)\cap S\) with \(\lambda< f(a)-f(c)\) and \(f(a)-f(c-)<\lambda+\mu\).
    \item For every \(\varepsilon>0\), there exists a finite subset \(F'\subset S\) of the form \(F'=\{a=x'_0<x'_1<\cdots<x'_k=b\}\)
    such that 
    \begin{equation*}
      \max_{1\leq i\leq k}\{f(x'_{i-1})-f(x'_{i}-)\}<\varepsilon.
    \end{equation*}
  \end{enumerate}
\end{lemma}

\begin{proof}
  We may assume that \(f(0)\neq 0\) since the claims are trivial when \(f(0)=0\).

  \medbreak

  (1)
  For \(a\in S\), we define a right-continuous non-decreasing function 
  \(h_a\colon \overline{S}\to \mathbb{R}\) by
  \begin{equation*}
    h_a(x):=f(a)-\overline{f}(x).
  \end{equation*}

  Set 
  \begin{equation*}
    A:=\{c\in[a,b)\cap\overline{S}\mid h_a(c)>\lambda\}.
  \end{equation*}
  Since \(\lambda<h_a(b-)\), the set \(A\) is non-empty and we define
  \begin{equation*}
    t:=\inf{A}.
  \end{equation*}
  Then by definition, we have \(h_a(t-)\leq\lambda\).
  
  First, suppose that \(h_a(t)>\lambda\).
  Since the set of left-jump points of \(h_a\) coincides with \(DP(\overline{f})\subset S\),
  \(t\in DP(\overline{f})\subset S\) and thus \(c=t\) satisfies the required conditions.
  Indeed, \(h_a(c-)\leq\lambda<\lambda+\mu\).

  Suppose \(h_a(t)\leq\lambda\). 
  Since \(h_a\) is right-continuous, 
  there exists \(\nu>0\) such that 
  \begin{equation*}
    x\in [t,t+\nu)\cap\overline{S}\implies h_a(x)<\lambda+\mu.
  \end{equation*}
  We may moreover take \(\nu>0\) so small that \(t+\nu<b\).
  Since \(t=\inf{A}\), 
  there exists 
  \begin{equation*}
    x\in [t,t+\nu)\cap A,
  \end{equation*}
  thus \(\lambda<h_a(x)<\lambda+\mu\).

  If \(x\in S\), then we put \(c:=x\), and \(c\) has the desired properties.
  Indeed, since \(h_a\) is non-decreasing, 
  \(h_a(c-)\leq h_a(c)<\lambda+\mu\).
  If \(x\notin S\), 
  then \(h_a\) has no jump at \(x\) 
  because 
  the set of left-jump points of \(h_a\) coincides with \(DP(\overline{f})\subset S\).
  Since \(h_a(x)>\lambda\), we can choose
  \(c\in S\) sufficiently close to \(x\) so that
  \begin{equation*}
    h_a(c-)\leq h_a(c)<\lambda+\mu.
  \end{equation*}

  Thus in all cases there exists \(c\in [a,b)\cap S\) such that
  \begin{equation*}
    \lambda<h_a(c)\quad\text{and}\quad h_a(c-)<\lambda+\mu.
  \end{equation*}

  \medbreak

  (2)
  If \(f(a)-f(b-)<\varepsilon\), then \(F':=\{a,b\}\) has the required property.
  We suppose that \(f(a)-f(b-)\geq\varepsilon\).
  Since \(f(a)-f(b-)<\infty\), 
  by repeatedly applying (1) with \(\lambda=\mu=\varepsilon/2\) finitely many times, 
  we obtain \(F'\) with the required property.

  We construct the required partition as follows.
  Let \(x_0:=a\).
  For \(i\geq 0\), 
  if \(f(x_i)-f(b-)\geq\varepsilon\), 
  then apply (1) to \([x_i,b)\) and obtain \(x_{i+1}\in(x_{i},b)\cap S\) satisfying the required properties.
  If \(f(x_i)-f(b-)<\varepsilon\), 
  then \(F':=\{x_j\}_{j=0}^{i}\cup\{b\}\) satisfies the claim.
  Since \(f(a)-f(b-)<\infty\) and \(f(a)-f(x_i)>i\varepsilon/2\),
  this procedure terminates after finitely many steps.
\end{proof}

\begin{definition}
  Let \((S,\leq)\) be a pre-ordered set, and let \(f\) be a function from \(S\) to \(\mathbb{C}\).
  We say that \(f\) \textit{vanishes at infinity} if, for all \(\varepsilon>0\), there exists \(x\in S\) such that 
  \begin{equation*}
    y\in S,\;x\leq y\implies \abs{f(y)}<\varepsilon.
  \end{equation*}
\end{definition}

\begin{lemma}
  \label{function:real:real:RCDappdbystepfunc}
  Let \(S\) be a subset of \(\mathbb{R}^+\) with \(0\in S\),
  let \(C>0\),
  and let \(f\colon S\to\mathbb{R}^+\) be a right-continuous decreasing function with \(DP(\overline{f})\subset S\).

  If \(f\) vanishes at infinity,    
  then there exists a sequence \(\{F_{n}\}_{n=0}^\infty\) of finite subsets of \(S\)
  such that \(0\in F_{n}\), \(F_{n}\subset F_{n+1}\) and 
  \begin{align*}
    \norm{g^{f}_{F_{n}}-f}< C/2^n, \qquad \text{for all \(n\in\mathbb{N}\).}
  \end{align*}
\end{lemma}

\begin{proof}
  We may assume that \(f(0)\neq 0\) since this claim is trivial when \(f(0)=0\).

  We prove this lemma by induction on \(n\).

  First consider the case \(n=0\). 
  Since \(f\) vanishes at infinity, 
  there exists \(b_0\in S\)
  with \(f(b_0)<C/2^{0}\).
  If \(b_0=0\), set \(F_0:=\{0\}\). 
  Then \(g^{f}_{F_0}=0\) and \(\norm{g^{f}_{F_0}-f}=f(0)<C\).
  If \(b_0>0\), applying the previous lemma (2) with \(a=0,b=b_0\), and \(\varepsilon=C\),
  we obtain a subset \(F_{0}\) of \(S\) such that
  \(\norm{g^{f}_{F_{0}}-f}<C/2^{0}\).

  Suppose this lemma holds for \(n=k\), that is, 
  there exists 
  a subset \(F_{k}\) of \(S\) such that \(\norm{g^{f}_{F_{k}}-f}<C/2^{k}\).
  Since \(f\) vanishes at infinity, 
  there exists \(b_{k+1}\in S\)
  such that \(\max{F_{k}}\leq b_{k+1}\) and \(f(b_{k+1})<C/2^{k+1}\).
  We write \(F'=F_{k}\cup\{b_{k+1}\}\) in the form
  \begin{equation*}
    F'=\{0=x'_{0}<x'_{1}<\cdots<x'_{\abs{F'}-1}\}\subset S,
  \end{equation*} 
  
  For each \(1\leq i \leq \abs{F'}-1\), applying the previous lemma (2) with \(a=x'_{i-1},b=x'_{i}\), and \(\varepsilon=C/2^{k+1}\), we obtain a finite subset \(F'_{i}\) of \(S\) with 
  \begin{equation*}
    F'_{i}=\{x'_{i-1}=x''_{i,0}<x''_{i,1}<\cdots<x''_{i,\abs{F'_{i}}-1}=x'_{i}\},
  \end{equation*}
  and \(\max_{1\leq j\leq \abs{F'_{i}}-1}\{f(x''_{i,j-1})-f(x''_{i,j}-)\}<C/2^{k+1}\).
  Thus 
  \begin{equation*}
    F_{k+1}:=F'\cup\bigcup_{1\leq i\leq \abs{F'}-1}F'_{i}\supset F_{k},
  \end{equation*}
  satisfies the required properties.
\end{proof}

\begin{lemma}
  \label{uniform_limit:RC:DP}
  Let \(S\) be a subset of \(\mathbb{R}^+\) containing \(0\), 
  and let \(\{f_n\colon S\to\mathbb{R}^+\}\)
  be a sequence of decreasing right-continuous functions 
  such that \(\{f_n\}\) uniformly converges to a function \(f\) on \(S\). 
  If \(DP(\overline{f_n})\subset S\) for all \(n\), 
  then \(f\) is decreasing right-continuous and \(DP(\overline f)\subset S\).
\end{lemma}

\begin{proof}
  The uniform limit of right-continuous functions is right-continuous.
  Moreover, \(f\) is decreasing as the pointwise limit of decreasing functions.

  It remains to prove 
  that \(DP(\overline{f})\subset S\) if \(DP(\overline{f_n})\subset S\) for all \(n\).

  Suppose that \(x\in DP(\overline{f})\), that is, \(\overline{f}(x-)-\overline{f}(x)>2\varepsilon\) for some \(\varepsilon>0\).
  Since \(\{f_n\}\) uniformly converges to \(f\), 
  there exists \(n\) with \(\norm{f-f_n}<\varepsilon/2\).
  By definition, for any \(x\in \overline{S}\setminus S\),
  \begin{equation*}
    \abs{\overline{f}(x)-\overline{f_n}(x)}\leq\norm{f-f_n}.
  \end{equation*}
  Thus \(\norm{\overline{f}-\overline{f_n}}_{\overline{S}}=\norm{f-f_n}\).
  We have
  \begin{align*}
    2\varepsilon
    &<\overline{f}(x-)-\overline{f}(x)\\
    &\leq
    \overline{f}(x-)-\overline{f_n}(x-)
    +\overline{f_n}(x-)-\overline{f_n}(x)
    +\overline{f_n}(x)-\overline{f}(x)\\
    &\leq 2\norm{f_n-f}+\overline{f_n}(x-)-\overline{f_n}(x).
  \end{align*}
  Therefore, \(\overline{f_n}(x-)-\overline{f_n}(x)>\varepsilon\) and \(x\in DP(\overline{f_n})\subset S\).
\end{proof}

\section{Definition and Basic Properties}

\label{section:SVF}

In this section 
we introduce the singular value functions for C\(^*\)-algebras 
and investigate their basic properties.

\subsection{Definition of singular value functions for C\(^*\)-algebras}

\label{subsection:Dfn_of_SVF}

The definition is motivated by the classical theory of singular values for compact operators on Hilbert spaces 
and its extension to semifinite von Neumann algebras with faithful traces.
In those settings, singular values (singular value functions) admit variational characterizations in terms of projections with bounded rank or trace.
Replacing rank or trace by the order structure of the positive cone of the \(K_0\)-group, we define singular value functions for C\(^*\)-algebras.

\begin{definition}
  \label{SVF:definition} 
  \label{SVF:Property:decreasing}
  Let \(A\) be a C\(^*\)-algebra. 
  For \(a\in A\), 
  we define a function \(s(a)\) on \(K_0(A)^+\) by  
  \begin{equation*}
    s_g(a):=\inf\{\norm{a-ap}\mid{p}\in{P(A)},\;[p]_0\leq g\}\in\mathbb{R}^+,
  \end{equation*}
  for each \(g\in K_0(A)^+\). 
  We call the function \(s(a)\) the \textit{singular value function} of \(a\). 
  It is clear that \(s(a)\) is nonnegative and decreasing.
  We also regard \(s\) as a map
  \begin{equation*}
    s \colon A \longrightarrow \{\, f \mid f \colon K_0(A)^+ \to \mathbb{R}^+ \,\}, \qquad a \longmapsto s(a).
  \end{equation*}
\end{definition}

\begin{remark}
  \label{SVF:extent:multiplier_algebra}
  Let \(A\) be a C\(^*\)-algebra.
  We can extend \(s\) to a map from the multiplier algebra \(M(A)\).
  For \(a\in M(A)\), the singular value function \(s(a)\) is defined by
  \begin{equation*}
    s_g(a):=\inf\{\norm{a-ap}\mid{p}\in{P(A)},[p]_0\leq g\}\in\mathbb{R}^+,
  \end{equation*}
  for each \(g\in K_0(A)^+\).
\end{remark}

\begin{remark}
  \label{SVF:extended:classical_SVS}
  For every \(a \in \mathcal{K}\), the singular value function \(s(a)\) coincides with the classical singular value sequence of \(a\).
  Indeed, the positive cone of the \(K_0\)-group of \(\mathcal{K}\) is naturally identified with \(\mathbb{N}\) via the rank map. 
  For \(a\in \mathcal{K}\) and \(n\in K_0(\mathcal{K})^+=\mathbb{N}\), we have 
  \begin{align*}
    s_n(a)
    &=\inf\{\norm{a-ap}\mid{p}\in{P(\mathcal{K})},[p]_0\leq n\}\\
    &=\inf\left\{\norm{a-ap}\mid p \text{ is a projection with } \mathrm{rank}(p)\leq n\right\}.
  \end{align*}
  Thus singular value sequences for \(\mathcal{K}\) are special cases of singular value functions for C\(^*\)-algebras. 
\end{remark}

\begin{proposition}
  \label{SVF:Property:vanishes_at_infty}
  Let 
  \(a\) be an element in a C\(^*\)-algebra \(A\).
  \begin{enumerate}
    \item If \(A\) is unital, then \(s_{[1_A]_0}(a)=0\).
    \item If \(A\) admits an approximate unit consisting of projections,
    then \(s(a)\) vanishes at infinity.
  \end{enumerate}
\end{proposition}

\begin{proof}
  (1) Trivial.
  (2)
  Since \(A\) has an approximate unit consisting of projections,
  there exists a projection \(q\in A\) with \(\norm{a-aq}<\varepsilon\).
  Set \(g:=[q]_0\). 
  For any \(h\in K_0(A)^+\) with \(g\leq h\), 
  we have
  \begin{align*}
    s_h(a)&=\inf\{\norm{a-ap}\mid p\in P(A), [p]_0\leq h\}\\
    &\leq\norm{a-aq}<\varepsilon.
  \end{align*}
  Thus the claim follows. 
\end{proof}

\begin{remark}
  The hypothesis of the previous proposition is satisfied, for example, by the following classes C\(^*\)-algebras:
  \begin{enumerate}
    \item C\(^*\)-algebras of real rank zero,
    \item stable stably unital C\(^*\)-algebras.
  \end{enumerate}
\end{remark}

\subsection{Singular value functions of positive elements with finite spectrum}

In this subsection, 
we determine the singular value functions of positive elements with finite spectrum.

First, we compute the singular value functions of projections.

\begin{lemma}
  \label{SVF:Example:projection}
  Let \(p\) be a projection in a C\(^*\)-algebra \(A\).
  Then we have
  \begin{equation*}
    s_g(p)=
    \begin{cases}
      1 & \text{if \([p]_0\not\leq g\),}\\
      0 & \text{if \([p]_0\leq g\),}
    \end{cases}
  \end{equation*}
  for any \(g\in K_0(A)^+\).
\end{lemma}

\begin{proof}
  Let \(g\) be an element in \(K_0(A)^+\).
  Since \(p\) is a projection, we have
  \begin{align*}
    0\leq s_g(p)\leq\norm{p}\norm{1_{\tilde{A}}-q}\leq1,
  \end{align*}
  where \(q\in P(A)\) with \([q]_0\leq g\). 
  If \([p]_0\leq g\), then \(s_g(p)\leq\norm{p-pp}=0\) and therefore \(s_g(p)=0\). 
  Suppose that \([p]_0\not\leq g\) and \(s_g(p)<1\).
  Then there exists a projection \(q\) such that \(\norm{p-pq}<1\) and \([q]_0\leq g\).
  By Lemma~\ref{subordinate:SC}, \([p]_0\leq[q]_0\) and 
  this is a contradiction.
  Thus we obtain \(s_g(p)=1\) for \([p]_0\not\leq g\).
\end{proof}

The following proposition says that the singular value functions of positive elements with finite spectrum are step functions. 

\begin{proposition}
  \label{SVF:Example:positive_elements_with_finite_spectrum}
  Let \(A\) be a C\(^*\)-algebra.
  If \(p_1,p_2,\dots,p_n\) are mutually orthogonal nonzero projections in \(A\) 
  and \(\alpha_0,\alpha_1,\dots,\alpha_{n-1}\) belong to \(\mathbb{R}^{+}\) with \(\alpha_0\geq\alpha_1\geq\cdots\geq\alpha_{n-1}\),
  then, for all \(g\in K_0(A)^+\), we have
  \begin{equation*}
    s_g\left(\sum_{i=1}^n\alpha_{i-1}p_i\right)
    =\min_{0\leq k\leq n}
    \left\{
      \alpha_k\mid\left[\hat{p}_k\right]_0\leq g
    \right\},
  \end{equation*}
  where \(\hat{p}_0=p_0=0\) and \(\alpha_n=0\).
\end{proposition}

\begin{proof}
  Set
  \begin{equation*}
    a:=\sum_{i=1}^n\alpha_{i-1}p_i.
  \end{equation*}

  For \(r\in P(A)\) and \(1\leq k\leq n\), 
  we have
  \begin{align*}
    \norm{{a}-{a} r}
    &\geq\norm{\hat{p}_k({a}-{a} r)}\\
    &\geq\norm{(\alpha_0p_1+\alpha_1p_2+\cdots+\alpha_{k-1}p_{k})(1_{\tilde{A}}-r)}\\
    &=\norm{(1_{\tilde{A}}-r)(\alpha_0^2p_1+\alpha_1^2p_2+\cdots+\alpha_{k-1}^2p_{k})(1_{\tilde{A}}-r)}^\frac{1}{2}\\
    &\geq\norm{(1_{\tilde{A}}-r)(\alpha_{k-1}^2p_1+\alpha_{k-1}^2p_2+\cdots+\alpha_{k-1}^2p_{k})(1_{\tilde{A}}-r)}^\frac{1}{2}\\
    &=\alpha_{k-1} \norm{\hat{p}_k-\hat{p}_k r}.
  \end{align*}
  Hence, we obtain 
  \(s({a})\geq \alpha_{k-1}s(\hat{p}_k)\) for all \(1\leq k\leq n\).
  In particular, 
  the previous lemma implies that 
  \begin{equation*}
    s_g({a})\geq \alpha_{k-1} \text{ if } \left[\hat{p}_{k}\right]_0\not\leq g.
  \end{equation*}

  For each \(0\leq k\leq n-1\), 
  if \(\left[\hat{p}_k\right]_0\leq g\), 
  then we obtain
  \begin{align*}
    s_g({a})
    &\leq\norm{{a}-{a}\hat{p}_k}\\
    &=\norm{\alpha_{k}p_{k+1}+\alpha_{k+1}p_{k+2}+\cdots+\alpha_{n-1}p_n}\\
    &\leq\alpha_{k}\norm{p_{k+1}+p_{k+2}+\cdots+p_n}\\
    &=\alpha_k.
  \end{align*}
  We have \(s_g({a})\leq\norm{{a}-{a}\hat{p}_n}=0\)
  if \(\left[\hat{p}_n\right]_0\leq g\).
  Thus for all \(0\leq k\leq n\), 
  we obtain
  \begin{equation*}
    s_g({a})\leq \alpha_k \text{ if \(\left[\hat{p}_k\right]_0\leq g\)}.
  \end{equation*}

  As a consequence, we obtain
  \begin{align*}
    s_g({a})
    &=
    \begin{cases}
      \alpha_k & \text{ if \(\left[\hat{p}_k\right]_0\leq g,\;\left[\hat{p}_{k+1}\right]_0\not\leq g\qquad(0\leq k\leq n-1)\)}\\
      \alpha_n & \text{ if \(\left[\hat{p}_n\right]_0\leq g\)}
    \end{cases}\\
    &=\min_{0\leq k\leq n}
    \left\{
      \alpha_k\mid\left[\hat{p}_k\right]_0\leq g
    \right\}.
  \end{align*}
  This completes the proof.
\end{proof}

\subsection{Basic properties of singular value functions for C\(^*\)-algebras}

In this subsection, we establish basic properties of singular value functions for C\(^*\)-algebras.
Under natural assumptions, properties analogous to those of the singular values of compact operators also hold.

\subsubsection{Fundamental properties of \(s\)}

\begin{proposition}
  \label{SVF:Basic_Property:1}
  Let \(A\) be a C\(^*\)-algebra. For \(a,b\in A\), the following hold.
  \begin{enumerate}
    \item If \(A\) is stably finite, then \(s_0(a)=\norm{a}\).
    \item \(s(\alpha a)=\abs{\alpha}s(a)\) for all \(\alpha\in\mathbb{C}\).
    \item \(s(ba)\leq \norm{b}s(a)\).
    \item \(\abs{s_g(a)-s_g(b)}\leq\norm{a-b}\) for all \(g\in K_0(A)^+\).
    \item \(s(a)=s(\abs{a})\).
    \item \(s(ua)=s(a),\;s(au)=s(a)\) for all \(u\in U(M(A))\).
    \item Let \(f\colon \mathbb{R}^+\to\mathbb{R}^+\) be a continuous increasing function with \(f(0)=0\).
    If \(a\) is positive and \(a\) has finite spectrum, then \(s(f(a))=f(s(a))\).
    \item If \(a^*a\leq b^*b\), then \(s(a)\leq s(b)\).
  \end{enumerate}
\end{proposition}

\begin{proof}
  (2), (3), (5) and (8) 
  are trivial.

  (1) 
  Since \(A\) is stably finite, \(\{p\in P(A)\mid[p]_0\leq0\}=\{0\}\).
  Therefore,
  \begin{equation*}
    s_0(a)=\inf\{\norm{a-ap}\mid{p}\in{P(A)},[p]_0\leq 0\}=\norm{a-a0}=\norm{a}.
  \end{equation*}

  \medbreak

  (4)
  Fix \(g\in K_0(A)^+\) and \(\varepsilon>0\).
  By definition of \(s_g(b)\), 
  there exists an element \(p\) in \(P(A)\) 
  such that \([p]_0\leq g\) and \(\norm{b-bp}<s_g(b)+\varepsilon\).
  Then we have,
  \begin{align*}
    s_g(a)
    &\leq\norm{a-ap}\\
    &\leq\norm{b-bp}+\norm{(a-b)-(a-b)p}\\
    &=\norm{b-bp}+\norm{(a-b)(1_{\tilde{A}}-p)}\\
    &<s_g(b)+\varepsilon+\norm{a-b}.
  \end{align*}
  Since \(\varepsilon>0\) is arbitrary, this implies
  \begin{equation*}
    s_g(a)-s_g(b)\leq\norm{a-b}.
  \end{equation*}
  Similarly, \(s_g(b)-s_g(a)\leq\norm{a-b}\). Thus we obtain \(\abs{s_g(a)-s_g(b)}\leq\norm{a-b}\).

  \medbreak

  (6) 
  It is clear that \(s(ua)=s(a)\).
  For any \(p\in P(A)\), 
  we have
  \begin{equation*}
    \norm{au-aup}
    =\norm{(a-aupu^*)u}
    =\norm{a-aupu^*},
  \end{equation*}
  and \([p]_0=[upu^*]_0\). 
  Thus we obtain \(s(au)=s(a)\).

  \medbreak

  (7)
  Since \(a\) is a positive element with finite spectrum, 
  there exist mutually orthogonal nonzero projections \(p_1,p_2,\dots,p_n\) in \(A\) and 
  \(\alpha_0,\alpha_1,\dots,\alpha_{n-1}\) in \(\mathbb{R}^+\) 
  with 
  \(\alpha_0\geq\alpha_1\geq\cdots\geq\alpha_{n-1}\)
  such that 
  \begin{equation*}
    a=\sum_{i=1}^{n}\alpha_{i-1}p_i.
  \end{equation*}
  Set \(p_0=0\) and \(\alpha_n=0\).
  Since \(f\) is increasing, \(\{f(\alpha_{i-1})\}_{i=1}^{n}\) is decreasing.
  By Proposition~\ref{SVF:Example:positive_elements_with_finite_spectrum}, 
  we obtain
  \begin{equation*}
    f(s_g(a))=\min_{0\leq k\leq n}
    \left\{
      f(\alpha_k)\mid\left[\hat{p}_k\right]_0\leq g
    \right\}.
  \end{equation*}
  and by the functional calculus: 
  \(f(a)=\sum_{i=1}^{n}f(\alpha_{i-1})p_i,\)
  we have
  \begin{equation*}
    s_g(f(a))=\min_{0\leq k\leq n}
    \left\{
      f(\alpha_k)\mid\left[\hat{p}_k\right]_0\leq g
    \right\}.
  \end{equation*}
  Thus we obtain \(s(f(a))=f(s(a))\).
\end{proof}

\begin{remark}
  \label{rmk:inMA:SVFBP1}
  The previous proposition except for (7) holds for singular value functions extended to elements in the multiplier algebras.
\end{remark}

\subsubsection{\(*\)-invariance of \(s\)}

In Corollary~\ref{SVF:Basic_Property:2}, 
we see that \(s\) is \(*\)-invariant on C\(^*\)-algebras of stable rank one. 

\begin{proposition}
  Let \(a\) be an element in a C\(^*\)-algebra \(A\).
  If \(a\) belongs to the closure of \(M(A)^{-1}\), then \(s(a)=s(a^*)\). 
\end{proposition}

\begin{proof}
  For arbitrary \(\varepsilon>0\), 
  there exists \(\tilde{a}\in M(A)^{-1}\) 
  such that \(\norm{a-\tilde{a}}<\varepsilon/2\).
  By Proposition~\ref{SVF:Basic_Property:1}~{(4)} and Remark~\ref{rmk:inMA:SVFBP1}, 
  for any \(g\in K_0(A)^+\), 
  we obtain
  \begin{equation*}
    \abs{s_g(a)-s_g(\tilde{a})}<\varepsilon/2 \text{ and } \abs{s_g(a^*)-s_g(\tilde{a}^*)}<\varepsilon/2.
  \end{equation*}
  By Proposition~\ref{SVF:Basic_Property:1}~{(6)} and polar decomposition, 
  we have \(s(\tilde{a})=s(\tilde{a}^*)\). 
  Therefore, we obtain
  \begin{align*}
    \abs{s_g(a)-s_g(a^*)}
    &\leq\abs{s_g(a)-s_g(\tilde{a})}+\abs{s_g(\tilde{a})-s_g(\tilde{a}^*)}+\abs{s_g(\tilde{a}^*)-s_g(a^*)}\\
    &<\varepsilon/2+\varepsilon/2=\varepsilon.
  \end{align*}
  This proves \(s(a)=s(a^*)\).
\end{proof}

The following corollary is a direct consequence of the previous proposition and Proposition~\ref{SVF:Basic_Property:1}~{(3)}.

\begin{corollary}
  \label{SVF:Basic_Property:2}
  Let \(A\) be a C\(^*\)-algebra. 
  If \(A\) is contained in the closure of \(M(A)^{-1}\), then \(s(a)=s(a^*)\) for all \(a\in A\).
  In particular, \(s(ab)\leq \norm{b}s(a)\) for all \(a,b\in A\).
\end{corollary}

\begin{example}
  Let \(A\) be a C\(^*\)-algebra.
  \begin{enumerate}
    \item If \(A\) has stable rank one, then \(\tilde{A}^{-1}\) is dense in \(\tilde{A}\). 
    In particular, \(A\) is contained in the closure of \(M(A)^{-1}\).
    Thus by Corollary~\ref{SVF:Basic_Property:2}, \(s\) is \(*\)-invariant on \(A\).
    \item If \(A\) is stable, then \(s\) is \(*\)-invariant on \(A\).
    Indeed, each element of \(A\) is approximated by elements of \(M(A)^{-1}\), as in \cite[Remark 1.8]{Thomsen:94:Inductive_limits_of_interval_algebras}.
    Thus Corollary~\ref{SVF:Basic_Property:2} also holds in this case.
  \end{enumerate}
\end{example}

\subsubsection{Analogues of Ky Fan's inequalities}

In Theorem~\ref{SVF:Basic_Property:3} and Theorem~\ref{SVF:Basic_Property:4},
we prove analogues of Ky Fan's inequalities \cite[(10), (11)]{KyFan:1951:inequalities_for_eigenvalues} of the singular value function on stable C\(^*\)-algebras of stable rank one.
In the following lemma, we approximate the sum of two projections by a single projection.

\begin{lemma}
  \label{approx:Proj2ByProj1}
  Let \(p\) and \(q\) be projections in a C\(^*\)-algebra \(A\) with cancellation.
  If \(A\) is contained in the closure of \({M(A)}^{-1}\)
  and \(K_0(A)^+=D_0(A)\), 
  then for any \(\varepsilon>0\),
  \begin{enumerate}
    \item there exists \(r\in P(A)\) with \([r]_0=[p\oplus q]_0 \text{ and } \norm{(p+q)-r(p+q)}<\varepsilon\),
    \item there exists \(r\in P(A)\) with \([r]_0=[p\oplus q]_0,\;
    \norm{p-pr}<\varepsilon\text{ and }\norm{q-qr}<\varepsilon\).
  \end{enumerate}
\end{lemma}

\begin{proof}
  (1)
  Take \(\varepsilon>0\) arbitrary and set \(\delta:=\varepsilon/3\).
  Since \(K_0(A)^+=D_0(A)=\{[p]_0\mid p\in P(A)\}\) 
  and \(A\) has cancellation, 
  there exists \(r'\in P(A)\) with \(r'\sim_{\mathrm{MvN}}p\oplus q\).
  Applying 
  \cite[Lemma~2.8~{(ii)}]{Kirchberg.Rordam:2000:non-simple.purely.infinite.C*-algebras}
  to \(r'\sim_{\mathrm{MvN}}p\oplus q\), 
  we obtain \(
    p+q\lesssim r'\).
  By \cite[Lemma~2.7~{(iii)}]{Kirchberg.Rordam:2000:non-simple.purely.infinite.C*-algebras}, 
  there exists a positive element \(r_0\) in the closed subalgebra \(r'Ar'\) of \(A\) such that \((p+q-\delta)_+\sim_{\mathrm{MvN}}r_0\): 
  there exists \(x\in A\) with \((p+q-\delta)_+=x^*x \text{ and } r_0=xx^*\).

  Since \(A\) is contained in the closure of \({M(A)}^{-1}\),
  there exists \(y\in {M(A)}^{-1}\)
  with
  \begin{align*}
    \norm{x-y}<(\norm{x}^2+\delta)^\frac{1}{2}-\norm{x}.
  \end{align*}
  Then, we have 
  \(\norm{y}\leq\norm{x}+\norm{y-x}<(\norm{x}^2+\delta)^\frac{1}{2}\)
  and
  \begin{align*}
    \norm{yy^*-r_0}
    &=\norm{yy^*-xx^*}=\norm{yy^*-yx^*+yx^*-xx^*}\\
    &\leq(\norm{y}+\norm{x})\norm{y-x}\\
    &<\left((\norm{x}^2+\delta)^\frac{1}{2}+\norm{x}\right)\left((\norm{x}^2+\delta)^\frac{1}{2}-\norm{x}\right)
    =\delta.
  \end{align*}
  Hence, we have \(\norm{yy^*-r_0}<\delta\)
  and similarly, \(\norm{(p+q-\delta)_+-y^*y}<\delta\). 
  Notice that \(yy^*\) is unitarily equivalent to \(y^*y\): 
  there exists \(u\in U(M(A))\) with \(y^*y=uyy^*u^*\). 
  
  Set \({r}:=u{r'}u^*\in P(A)\). We have
  \begin{align*}
    [r]_0=[r']_0=[p\oplus q]_0.
  \end{align*}
  It remains to show \(\norm{(p+q)-r(p+q)}<\varepsilon\).
  Using \((1_{M(A)}-r')r_0=0\), we have
  \begin{align*}
    (1_{M(A)}-r)y^*y
    &=u(1_{M(A)}-{r'})u^*uyy^*u^*\\
    &=u(1_{M(A)}-{r'})yy^*u^*\\
    &=u(1_{M(A)}-{r'})(yy^*-r_0)u^*.
  \end{align*}
  In particular, \(\norm{(1_{M(A)}-r)y^*y}\leq\norm{yy^*-r_0}<\delta\).
  Hence we obtain
  \begin{align*}
    &\norm{(p+q)-{r}(p+q)}\\
    &=\norm{(1_{M(A)}-r)(p+q)}\\
    &=\norm{
      (1_{M(A)}-r)
      \left(
        (p+q)-(p+q-\delta)_++(p+q-\delta)_+-y^*y+y^*y
      \right)
    }\\
    &\leq\norm{(1_{M(A)}-r)((p+q)-(p+q-\delta)_+)}\\
    &\quad+\norm{(1_{M(A)}-r)((p+q-\delta)_+-y^*y)}\\
    &\qquad+\norm{(1_{M(A)}-r)y^*y}\\
    &\leq\norm{(p+q)-(p+q-\delta)_+}+\norm{(p+q-\delta)_+-y^*y}+\norm{yy^*-r_0}\\
    &<3\delta=\varepsilon.
  \end{align*}

  \medbreak

  (2)
  By (1), there exists a projection \(r\) in \(A\) such that 
  \begin{align*}
    [r]_0=[p\oplus q]_0 \text{ and } \norm{(p+q)-r(p+q)}<\varepsilon^2.
  \end{align*}
  Using \(p\leq p+q\), we have
  \begin{align*}
    \norm{p-pr}^2
    &=\norm{(1_{M(A)}-r)p(1_{M(A)}-r)}\\
    &\leq\norm{(1_{M(A)}-r)(p+q)(1_{M(A)}-r)}\\
    &\leq\norm{(1_{M(A)}-r)(p+q)}<\varepsilon^2.
  \end{align*}
  We thus get \(\norm{p-pr}<\varepsilon\) and similarly \(\norm{q-qr}<\varepsilon\).
\end{proof}

The following theorem is one of the main results of this paper.

\begin{theorem}
  \label{SVF:Basic_Property:3}
  Let \(A\) be a C\(^*\)-algebra with cancellation.
  If \(A\) is contained in the closure of \({M(A)}^{-1}\) and \(K_0(A)^+=D_0(A)\), then 
  \begin{align*}
    s_{g+h}(a+b)\leq s_g(a)+s_h(b),
  \end{align*}
  for \(a,b\in A,\;g,h\in K_0(A)^+\).
\end{theorem}

\begin{proof}

  For each \(\varepsilon>0\)
  there exist \(p,q\in P(A)\) such that 
  \begin{align*}
    [p]_0\leq g,\;[q]_0\leq h,\; \norm{a-ap}\leq s_g(a)+\varepsilon \text{ and } \norm{b-bq}\leq s_h(b)+\varepsilon.
  \end{align*}
  By Lemma~\ref{approx:Proj2ByProj1}~{(2)}, 
  there exists \(r\in P(A)\) 
  such that 
  \begin{align*}
    \norm{p-pr}<\varepsilon/2,\; \norm{q-qr}<\varepsilon/2 \text{ and } [r]_0=[p\oplus q]_0.
  \end{align*}
  Hence, we have
  \begin{align*}
    \norm{p-rpr}=\norm{p-pr+(p-rp)r}\leq\norm{p-pr}+\norm{(p-pr)^*}\norm{r}<\varepsilon.
  \end{align*}

  Considering \(rpr\leq r\), in particular, \(1_{M(A)}-r\leq1_{M(A)}-rpr\), we have
  \begin{align*}
    \norm{a-ar}^2
    &=\norm{a(1_{M(A)}-r)a^*}\\
    &\leq\norm{a(1_{M(A)}-rpr)a^*}\\
    &\leq\norm{a(1_{M(A)}-p)a^*}+\norm{a(p-rpr)a^*}\\
    &<(s_g(a)+\varepsilon)^2+\norm{a}^2\varepsilon.
  \end{align*}
  Similarly, we have 
  \(\norm{b-br}^2<(s_h(b)+\varepsilon)^2+\norm{b}^2\varepsilon\).
  Therefore, 
  \begin{align*}
    s_{g+h}(a+b)
    &\leq\norm{(a+b)-(a+b)r}\\
    &\leq\norm{a-ar}+\norm{b-br}\\
    &\leq\left((s_g(a)+\varepsilon)^2+\norm{a}^2\varepsilon\right)^{1/2}+\left((s_h(b)+\varepsilon)^2+\norm{b}^2\varepsilon\right)^{1/2}.
  \end{align*}
  Letting \(\varepsilon\to 0\), we obtain
  \begin{equation*}
    s_{g+h}(a+b)\leq s_g(a)+s_h(b).
  \end{equation*}
\end{proof}

\begin{remark}
  The previous theorem holds for C\(^*\)-algebras satisfying certain favorable conditions.
  \begin{enumerate}
    \item For a stable C\(^*\)-algebra with cancellation, the previous theorem holds.
    Indeed, if \(A\) is stable, then \(K_0(A)^+=D_0(A)\) 
    and \(A\) is contained in the closure of \(M(A)^{-1}\).
    \item If a C\(^*\)-algebra has stable rank one and \(K_0(A)^+=D_0(A)\), then the previous theorem holds.
    In fact, being of stable rank one implies having cancellation and \(A\) is contained in the closure of \(\tilde{A}^{-1}\).
  \end{enumerate}
\end{remark}

The following theorem is one of the main results of this paper.

\begin{theorem}
  \label{SVF:Basic_Property:4}
  Let \(A\) be a C\(^*\)-algebra with cancellation.
  If \(A\) has real rank zero and \(K_0(A)^+=D_0(A)\), then 
  \begin{align*}
    s_{g+h}(ab)\leq s_g(a)s_h(b),
  \end{align*}
  for \(a,b\in A,\;g,h\in K_0(A)^+\).
\end{theorem}

\begin{proof}
  For any \(\varepsilon>0\),
  there exist \(p_1,p_2\in P(A)\) such that 
  \begin{equation}
    \label{ineq:KyFan6:SVFof2}
    [p_1]_0\leq g,\; [p_2]_0\leq h, \;\norm{a-ap_1}\leq s_g(a)+\varepsilon,\; \norm{b-bp_2}\leq s_h(b)+\varepsilon.
  \end{equation}

  We set a positive element \(r:=b^*p_1b\).
  Since \(A\) has real rank zero, the hereditary subalgebra \(\overline{rAr}\) of \(A\) has an approximate unit that consists of projections. 
  Hence there exists a projection \(q_\varepsilon\) in \(\overline{rAr}\subset A\) such that \(\norm{r-r{q_{\varepsilon}}}\leq\varepsilon^2\).
  Then we obtain 
  \begin{equation}
    \label{ineq:KyFan6:hered}
    \norm{p_1b(1_{\tilde{A}}-{q_{\varepsilon}})}<\varepsilon.
  \end{equation}
  Indeed,
  \begin{align*}
    \norm{p_1b(1_{\tilde{A}}-{q_{\varepsilon}})}^2
    &=\norm{(1_{\tilde{A}}-{q_{\varepsilon}})r(1_{\tilde{A}}-{q_{\varepsilon}})}\\
    &\leq\norm{r(1_{\tilde{A}}-{q_{\varepsilon}})}\leq\varepsilon^2.
  \end{align*}

  Since \(A\) has real rank zero and cancellation, \(A\) has stable rank one and thus \(A\) is contained in the closure of \(\tilde{A}^{-1}\).
  By Lemma~\ref{approx:Proj2ByProj1}~{(2)}, 
  there exists a projection \(p\) in \(A\)
  such that 
  \begin{equation}
    \label{ineq:KyFan6:Proj2ByProj1}
    [p]_0=[{q_{\varepsilon}}]_0+[p_2]_0, \; \norm{{q_{\varepsilon}}-{q_{\varepsilon}}p}<\varepsilon \text{ and } \norm{p_2-p_2p}<\varepsilon.
  \end{equation}
  We have
  \begin{align*}
    &ab(1_{\tilde{A}}-p)-a(1_{\tilde{A}}-p_1)b(1_{\tilde{A}}-p_2)(1_{\tilde{A}}-p)\\
    &=a
    \left(
      b-(1_{\tilde{A}}-p_1)b(1_{\tilde{A}}-p_2)
    \right)
    (1_{\tilde{A}}-p)\\
    &=a
    \left(
      b-b+p_1b-p_1bq_\varepsilon+p_1bq_\varepsilon
      +(1_{\tilde{A}}-p_1)bp_2
    \right)
    (1_{\tilde{A}}-p)\\
    &=a
    \left(
      p_1b(1_{\tilde{A}}-{q_{\varepsilon}})+p_1b{q_{\varepsilon}}+(1_{\tilde{A}}-p_1)bp_2
    \right)
    (1_{\tilde{A}}-p)\\
    &=ap_1b(1_{\tilde{A}}-{q_{\varepsilon}})(1_{\tilde{A}}-p)+ap_1b{q_{\varepsilon}}(1_{\tilde{A}}-p)+a(1_{\tilde{A}}-p_1)bp_2(1_{\tilde{A}}-p).
  \end{align*}
  Thus we obtain
  \begin{align*}
    \norm{ab-abp}
    &\leq\norm{a(1_{\tilde{A}}-p_1)b(1_{\tilde{A}}-p_2)(1_{\tilde{A}}-p)}+\norm{ap_1b(1_{\tilde{A}}-{q_{\varepsilon}})(1_{\tilde{A}}-p)}\\
    &\quad+\norm{ap_1b{q_{\varepsilon}}(1_{\tilde{A}}-p)}+\norm{a(1_{\tilde{A}}-p_1)bp_2(1_{\tilde{A}}-p)}\\
    &\leq\norm{a(1_{\tilde{A}}-p_1)}\norm{b(1_{\tilde{A}}-p_2)}\norm{1_{\tilde{A}}-p}+\norm{a}\norm{p_1b(1_{\tilde{A}}-{q_{\varepsilon}})}\norm{1_{\tilde{A}}-p}\\
    &\quad+\norm{a}\norm{p_1}\norm{b}\norm{{q_{\varepsilon}}(1_{\tilde{A}}-p)}+\norm{a}\norm{1_{\tilde{A}}-p_1}\norm{b}\norm{p_2(1_{\tilde{A}}-p)}\\
    &\leq\norm{a(1_{\tilde{A}}-p_1)}\norm{b(1_{\tilde{A}}-p_2)}+\norm{a}\norm{p_1b(1_{\tilde{A}}-{q_{\varepsilon}})}\\
    &\quad+\norm{a}\norm{b}\norm{{q_{\varepsilon}}(1_{\tilde{A}}-p)}+\norm{a}\norm{b}\norm{p_2(1_{\tilde{A}}-p)}\\
    (\text{by}~\eqref{ineq:KyFan6:SVFof2})\qquad
    &\leq(s_g(a)+\varepsilon)(s_h(b)+\varepsilon)+\norm{a}\norm{p_1b(1_{\tilde{A}}-{q_{\varepsilon}})}\\
    &\quad+\norm{a}\norm{b}\norm{{q_{\varepsilon}}(1_{\tilde{A}}-p)}+\norm{a}\norm{b}\norm{p_2(1_{\tilde{A}}-p)}\\
    (\text{by}~\eqref{ineq:KyFan6:hered})\qquad
    &\leq(s_g(a)+\varepsilon)(s_h(b)+\varepsilon)+\norm{a}\varepsilon\\
    &\quad+\norm{a}\norm{b}\norm{{q_{\varepsilon}}(1_{\tilde{A}}-p)}+\norm{a}\norm{b}\norm{p_2(1_{\tilde{A}}-p)}\\
    (\text{by}~\eqref{ineq:KyFan6:Proj2ByProj1})\qquad
    &<(s_g(a)+\varepsilon)(s_h(b)+\varepsilon)+\norm{a}\varepsilon+2\norm{a}\norm{b}\varepsilon.
  \end{align*}

  Since \({q_{\varepsilon}}\) is a projection in \(\overline{rAr}\),
  we have \({q_{\varepsilon}}\lesssim r\) by \cite[Lemma~2.7~(i)]{Kirchberg.Rordam:2000:non-simple.purely.infinite.C*-algebras} and \(r\lesssim p_1\) by definition. 
  Thus we obtain \({q_{\varepsilon}}\lesssim p_1\),
  in particular, \([{q_{\varepsilon}}]_0\leq[p_1]_0\).

  We also have \([p]_0=[q_\varepsilon]_0+[p_2]_0\) by \eqref{ineq:KyFan6:Proj2ByProj1}.
  Since \([p]_0\leq[p_1]_0+[p_2]_0 \leq g+h\),
  \begin{equation*}
    s_{g+h}(ab)\leq (s_g(a)+\varepsilon)(s_h(b)+\varepsilon)+\norm{a}\varepsilon+2\norm{a}\norm{b}\varepsilon.
  \end{equation*}
  Letting \(\varepsilon\to0\), we obtain 
  \begin{equation*}
    s_{g+h}(ab)\leq s_g(a)s_h(b)
  \end{equation*}
  as desired.
\end{proof}

\begin{remark}
  For a stable C\(^*\)-algebra \(A\) of real rank zero and with cancellation, the previous theorem holds.
  In fact, \(K_0(A)^+=D_0(A)\) since \(A\) is stable.
\end{remark}

\subsubsection{Monotonicity}

For real rank zero C\(^*\)-algebras,
we can extend 
Proposition~\ref{SVF:Basic_Property:1}~{(7)} to all positive elements.

\begin{proposition}
  \label{SVF:Basic_Property:5}
  Let \(a\) be a positive element in a real rank zero C\(^*\)-algebra \(A\)
  and \(f\colon \mathbb{R}^+\to\mathbb{R}^+\) be a continuous increasing function
  with \(f(0)=0\).
  Then 
  \begin{equation*}
    s(f(a))=f(s(a)).
  \end{equation*}
\end{proposition}

\begin{proof}
  Since \(A\) has real rank zero,
  there exists a sequence of positive elements \(\{{\tilde{a}_n}\}_{n\in\mathbb{N}}\) in \(A\) 
  with finite spectrum and \(\norm{a-{\tilde{a}_n}}<1/(n+1)\).
  Since \(s(a)\) and \(s({\tilde{a}_n})\) are nonnegative and decreasing, 
  \(f\) is uniformly continuous on their bounded ranges.
  Note that 
  \begin{align*}
    \norm{f(s({\tilde{a}_n}))-f(s(a))}&=\sup_{g\in K_0(A)^+}\abs{f(s_g({\tilde{a}_n}))-f(s_g(a))},\\
    \abs{s_g({\tilde{a}_n})-s_g(a)}&\leq\norm{s({\tilde{a}_n})-s(a)}\leq\norm{{\tilde{a}_n}-a}<1/(n+1),
  \end{align*}
  by Proposition~\ref{SVF:Basic_Property:1}~{(4)}.
  Fix \(\varepsilon>0\).
  There exists \(n_1\in\mathbb{N}\) such that
  \begin{equation*}
    \norm{f(s({\tilde{a}_n}))-f(s(a))}<\varepsilon, \text{ if } n\geq n_1.
  \end{equation*}
  Set \(M:=\sup_{n\in\mathbb{N}}\norm{\tilde{a}_n}\).
  There exists a polynomial \(\tilde{f}\) 
  with \(\tilde{f}(0)=0\), \(\norm{f-\tilde{f}}_{[0,M]}<\varepsilon/3\),
  and there exists \(n_2\in\mathbb{N}\) 
  such that 
  \(\norm{\tilde{f}(a)-\tilde{f}(\tilde{a}_n)}<\varepsilon/3\)
  if \(n\geq n_2\).
  By Proposition~\ref{SVF:Basic_Property:1}~{(4)} and the functional calculus, 
  if \(n\geq n_2\), then
  \begin{align*}
    \norm{s(f(a))-s(f({\tilde{a}_n}))}
    &\leq\norm{f(a)-f({\tilde{a}_n})}\\
    &\leq\norm{f(a)-\tilde{f}(a)}+\norm{\tilde{f}(a)-\tilde{f}(\tilde{a}_n)}+\norm{\tilde{f}(\tilde{a}_n)-f({\tilde{a}_n})}
    <\varepsilon.
  \end{align*}
  
  By Proposition~\ref{SVF:Basic_Property:1}~{(7)}, if \(n\geq \max\{n_1,n_2\}\), then
  \begin{equation*}
    \norm{s(f(a))-f(s(a))}\leq\norm{s(f(a))-s(f({\tilde{a}_n}))}+\norm{f(s({\tilde{a}_n}))-f(s(a))}<2\varepsilon.
  \end{equation*}
  Letting \(\varepsilon\to 0\), we obtain \(s(f(a))=f(s(a))\).
\end{proof}

As an application of the previous proposition, the following holds.

\begin{proposition}
  \label{SVF:Basic_Property:6}
  Let \(a\) and \(b\) be elements in a C\(^*\)-algebra \(A\) of real rank zero.
  The following holds.
  \begin{enumerate}
    \item If \(s(a^*)=s(a)\), then \(s(a^*a)=s(aa^*)\).
    \item For \(a, b\in A^+\), if \(s(b^{1/2}a^{1/2})=s(a^{1/2}b^{1/2})\), then \(s(a^{1/2}ba^{1/2})=s(b^{1/2}ab^{1/2})\).
    \item If \(0\leq a\leq b\), then \(s(a)\leq s(b)\).
  \end{enumerate}
\end{proposition}

\begin{proof}
  (1)
  By Proposition~\ref{SVF:Basic_Property:1}~{(5)} and Proposition~\ref{SVF:Basic_Property:5}, 
  \begin{equation*}
    s(a^*a)=s(\abs{a}^2)=s(\abs{a})^2=s(a)^2=s(a^*)^2=s(\abs{a^*})^2=s(\abs{a^*}^2)=s(aa^*).
  \end{equation*}

  \medbreak

  (2)
  Since \(a^{1/2}ba^{1/2}=(b^{1/2}a^{1/2})^*(b^{1/2}a^{1/2})\),
  this is a direct consequence of \((1)\).

  \medbreak

  (3)
  It follows from Proposition~\ref{SVF:Basic_Property:1}~{(8)} that \(s(a^{1/2})\leq s(b^{1/2})\).
  Using Proposition~\ref{SVF:Basic_Property:5}, we have
  \(s(a)=s(a^{1/2})^2\leq s(b^{1/2})^2=s(b)\).
\end{proof}

The singular value functions for C\(^*\)-algebras of real rank zero have good properties.

\begin{example}
  If \(a\) is an element in a C\(^*\)-algebra \(A\) of real rank zero, then \(s(a)\) is decreasing, vanishes at infinity, and is a uniform limit of step functions.

  Indeed, 
  by Proposition~\ref{SVF:Property:vanishes_at_infty}, \(s(a)\) vanishes at infinity.
  By Proposition~\ref{SVF:Example:positive_elements_with_finite_spectrum}, Proposition~\ref{SVF:Basic_Property:1}~{(4)}~and~{(5)}, \(s(a)=s(\abs{a})\) is a uniform limit of step functions.
\end{example}

\section{Realization and Continuity of Singular Value Functions}

\label{section:realization-continuity}

In this section, 
we study realization and continuity properties of singular value functions for C\(^*\)-algebras.

\subsection{Singular value functions for C\(^*\)-algebras with totally ordered \(K_0\)-groups}

In this subsection, we characterize the singular value functions for certain classes of C\(^*\)-algebras.

It is easy to see that a totally ordered group is simple if and only if it is Archimedean.
If the ordered group \((K_0(A),K_0(A)^+)\) is simple and totally ordered, 
we may regard \((K_0(A),K_0(A)^+)\) as a subset of \((\mathbb{R},\mathbb{R}^+)\).
In this case, one can formulate the notion of right-continuity for functions on \(K_0(A)^+\), 
just as for singular value functions of elements in semifinite von Neumann algebras with faithful traces.
Under these assumptions, we obtain the following theorem.

\begin{theorem}
  \label{SVF:form:simple_and_totally_ordered}
  Let \(A\) be a C\(^*\)-algebra with cancellation such that the ordered \(K_0\)-group of \(A\) is simple and totally ordered.
  If \(A\) is unital (respectively stable), then the following hold:
  \begin{enumerate}
    \item Let \(f\colon K_0(A)^+ \to \mathbb{R}^+\) be a decreasing, right-continuous function such that \(DP(\overline{f})\subset K_0(A)^+\) and \(f([1_A]_0)=0\) (respectively vanishes at infinity). There exists \(a \in A\) such that \(f = s(a)\).
    \item If \(A\) has real rank zero, then for every \(a\in A\), the singular value function \(s(a)\) is right-continuous, vanishes at infinity, and \(DP(\overline{s(a)})\subset K_0(A)^+\).
  \end{enumerate}
\end{theorem} 

\begin{proof}
  (1) 
  Since the case \(f=0\) is trivial, we may assume that \(f\) is not the zero function. 

  Let \(g_0=0\) and \(\{g_n\}_{n=0}^N\) be an increasing sequence in \(K_0(A)^+\)
  with a set of mutually orthogonal projections \(\{p_n\}_{n=1}^{N}\) in \(A\) 
  such that \([\hat{p}_{n}]_0=g_n\) for all \(1\leq n\leq N\). 
  Define \(a(\{g_n\},\{p_n\})\in A\) by
  \begin{equation*}
    a(\{g_n\},\{p_n\})=\sum_{n=1}^{N}(f(g_{n-1}))p_n.
  \end{equation*}
  Since \(f\) is decreasing and \(K_0(A)^+\subset\mathbb{R}\), Proposition~\ref{SVF:Example:positive_elements_with_finite_spectrum} implies that \(s(a(\{g_n\},\{p_n\}))\) is a right-continuous decreasing step function on \(K_0(A)^+\) such that
  \begin{align*}
    s(a(\{g_n\},\{p_n\}))
    &=s
    \left(
      \sum_{n=1}^{N}(f(g_{n-1})p_n)
    \right)
    =\sum_{n=1}^{N}f(g_{n-1})
    \chi_{\left[[\hat{p}_{n-1}]_0,[\hat{p}_{n}]_0\right)}\\
    &=\sum_{n=1}^{N}f(g_{n-1})\chi_{[g_{n-1},g_{n})}
    =g^{f}_{\{g_k\}}.
  \end{align*}
  Let \(h_0=0\) and \(\{h_m\}_{m=0}^M\) be an increasing sequence in \(K_0(A)^+\)
  with a set of mutually orthogonal projections \(\{q_m\}_{m=1}^{M}\) in \(A\) 
  such that \([\hat{q}_m]_0=h_m\) for all \(1\leq m\leq M\) and there exists \(1\leq m_1<\cdots<m_N\leq M\) with
  \(\hat{q}_{m_n}=\hat{p}_{n}\text{ for all \(1\leq n\leq N\)}\).
  Then 
  \begin{equation}
    \label{difference_of_operators_gen_by_steps}
    \norm{a(\{g_n\},\{p_n\})-a(\{h_m\},\{q_m\})}\leq \norm{g_{\{g_n\}}^{f}-f}.
  \end{equation}
  In fact, by \(p_n=\sum_{i=m_{n-1}+1}^{m_n}q_i\), we have
  \begin{align*}
    &a(\{g_n\},\{p_n\})-a(\{h_m\},\{q_m\})\\
    &=
    \sum_{n=1}^{N}\sum_{i=m_{n-1}+1}^{m_n}
    (f(g_{n-1})-f(h_{i-1}))q_i
    -
    \sum_{i=m_{N}+1}^{M}f(h_{i-1})q_i.
  \end{align*}
  For each \(1\leq n\leq N, m_{n-1}+1\leq i\leq m_{n}\),
  \begin{align*}
    \abs{f(g_{n-1})-f(h_{i-1})}\leq \abs{f(g_{n-1})-f(g_{n}-)}\leq\norm{g_{\{g_n\}}^{f}-f}.
  \end{align*}
  For each \(m_{N}+1\leq i\leq M\),
  \begin{align*}
    \abs{f(h_{i-1})}\leq\abs{f(g_{N})}\leq\norm{g_{\{g_n\}}^{f}-f}.
  \end{align*}
  Thus 
  \begin{align*}
    &\norm{a(\{g_n\},\{p_n\})-a(\{h_m\},\{q_m\})}\\
    &=
    \max
    \left(
      \max_{1\leq n\leq N, m_{n-1}+1\leq i\leq m_{n}}\abs{f(g_{n-1})-f(h_{i-1})}, 
      \max_{m_{N}+1\leq i\leq M}\abs{f(h_{i-1})}
    \right)\\
    &\leq\norm{g_{\{g_n\}}^{f}-f}.
  \end{align*}
  Hence we obtain the required inequality.

  By Lemma~\ref{function:real:real:RCDappdbystepfunc}, 
  there exists an increasing sequence \(\{F_n\}_{n\in\mathbb{N}}\) of finite subsets in \(K_0(A)^+\) such that, for all \(n\in\mathbb{N}\), \(0\in F_n\) and
  \begin{equation*}
    \norm{g_{F_n}^f-f}<f(0)/2^{n}.
  \end{equation*}
  If \(A\) is unital, since \(f([1_A]_0)=0\),
  we suppose 
  \begin{equation*}
    F_n\subset\{g\in K_0(A)^+\mid 0\leq g\leq [1_A]_0\}, \quad\text{for all \(n\in\mathbb{N}\)}.
  \end{equation*}

  Write
  \begin{equation*}
    F_n=\{0=g_{0}^{(n)}<g_{1}^{(n)}<\cdots<g_{\abs{F_n}-1}^{(n)}\}\quad\text{for all \(n\in\mathbb{N}\)}.
  \end{equation*}
 If \(A\) is stable, 
 then there exists \(q^{(n)}\in P(A)\) with \([q^{(n)}]_0=g_{\abs{F_n}-1}^{(n)}\). 
 If \(A\) is unital, 
 then, since \(A\) has cancellation and \(g_{\abs{F_n}-1}^{(n)}\leq[1_A]_0\),
 there exists \(q^{(n)}\in P(A)\) with \([q^{(n)}]_0=g_{\abs{F_n}-1}^{(n)}\) by \cite[Lemma~3.2~{(1)}]{Rordam:2004:Stable_C*-algebras}. 

  To prove (1), we construct a Cauchy sequence \(\{a_n\}_{n=0}^{\infty}\) in \(A\) such that 
  \begin{equation*}
    \norm{a_n-a_{n-1}}<{f(0)}/{2^{n-1}}\quad\text{and}\quad\norm{s(a_n)-f}<{f(0)}/{2^n},
  \end{equation*}
  for any \(n\in\mathbb{N}\) with \(n\geq 1\).

  \medskip

  \textbf{[Step 1.]} 
  We first consider the case \(n=0\).
  By Lemma~\ref{projection:depart:sequence}~(1), 
  we obtain an increasing sequence \(\{p'_k\}_{k=1}^{\abs{F_0}-1}\) in \(P(A)\) 
  with
  \([p'_k]_0=g_k^{(0)}\)
  for \(1\leq k\leq \abs{F_0}-1\).
  Set \(p_k^{(0)}:=p'_k-p'_{k-1}\) for each \(1\leq k\leq \abs{F_0}-1\), where \(p'_0=0\). 
  Then \(\{p_k^{(0)}\}_{k=1}^{\abs{F_0}-1}\) is a set of mutually orthogonal projections in \(A\) 
  with
  \([\hat{p}^{(0)}_k]_0=g_k^{(0)}\)
  for \(1\leq k\leq \abs{F_0}-1\).
  Put \(a_{0}:=a(\{g_n^{(0)}\},\{p_n^{(0)}\})\).

  \medskip

  \textbf{[Step 2.]} Next we construct \(a_1\).
  Since \(A\) has cancellation, 
  using \cite[Proposition 2.2.2]{Rordam.Larsen.Laustsen} as in Lemma~\ref{Cuntz-subeq2subproj},
  we obtain \(q'\in P(A)\) 
  such that \(q'\sim_{\mathrm{MvN}}q^{(1)},[q']_0=g_{\abs{F_1}-1}^{(1)}\) and \(\hat{p}^{(0)}_{\abs{F_0}-1}\leq q'\).
  Applying Lemma~\ref{projection:depart:sequence}~(2) to \(\{g_k^{(1)}\}_{k=1}^{\abs{F_1}-1}\) and \(\{\hat{p}^{(0)}_{k}\}_{k=1}^{\abs{F_0}}\), 
  where \(\hat{p}^{(0)}_{\abs{F_0}}:=q'\),
  we obtain an increasing sequence \(\{q_k^{(1)}\}_{k=1}^{\abs{F_1}-1}\) in \(P(A)\)
  such that 
  \begin{equation*}
    \{\hat{p}_{k}^{(0)}\}_{k=1}^{\abs{F_0}}\subset\{q_k^{(1)}\}_{k=1}^{\abs{F_1}-1} \text{ and } [q_k^{(1)}]_0=g_k^{(1)} \text{ for all } 1\leq k\leq \abs{F_1}-1.
  \end{equation*}
  Set \(p_{k}^{(1)}:=q_k^{(1)}-q_{k-1}^{(1)}\)
  for each \(1\leq k\leq \abs{F_1}-1\), where \(q_0^{(1)}=0\).
  We obtain mutually orthogonal projections \(\{p_k^{(1)}\}_{k=1}^{\abs{F_1}-1}\) in \(A\) 
  with \([\hat{p}_{k}^{(1)}]_0=g_k^{(1)}\)
  for \(1\leq k\leq \abs{F_1}-1\). 
  Put \(a_{1}:=a(\{g_n^{(1)}\},\{p_n^{(1)}\})\).

  By \(\eqref{difference_of_operators_gen_by_steps}\), we obtain
  \begin{equation*}
    \norm{a_{0}-a_{1}}\leq\norm{g^{f}_{F_0}-f}<{f(0)}/{2^{0}}.
  \end{equation*}

  \medskip

  \textbf{[Step 3.]} 
  Repeating these processes, 
  we obtain 
  \(F_n\) such that \(\norm{g^{f}_{F_n}-f}<f(0)/2^{n}\)
  and \(\{a_n\}\) such that 
  \begin{equation*}
    \norm{a_{n-1}-a_n}<\norm{g^{f}_{F_{n-1}}-f}<f(0)/2^{n-1},\; s(a_n)=g^{f}_{F_n}.
  \end{equation*}
  Thus we have a Cauchy sequence \(\left\{a_{n}\right\}\) in \(A\) 
  and hence \(\left\{a_{n}\right\}\) converges to an element \(a\in A\).
  By Proposition~\ref{SVF:Basic_Property:1}~(4), we have
  \begin{align*}
    \norm{s(a)-f}
    &\leq\norm{s(a)-s(a_n)}+\norm{s(a_n)-f}\\
    &\leq\norm{a-a_n}+\norm{g^{f}_{F_n}-f}.
  \end{align*}
  Letting \(n\to\infty\), we obtain \(s(a)=f\).

  \medbreak

  (2) 
  We may assume that \(a\in A^+\) by Proposition~\ref{SVF:Basic_Property:1}~(5).

  Let \(\varepsilon>0\).
  Since \(A\) has real rank zero, 
  for each \(a\in A^+\), there is \(\tilde{a}\in A^+\) with finite spectrum and \(\norm{a-\tilde{a}}<\varepsilon/3\).
  By Proposition~\ref{SVF:Example:positive_elements_with_finite_spectrum}, 
  for \(\tilde{a}\in A^+\) with finite spectrum, 
  \(s(\tilde{a})\) is right-continuous: 
  for any \(g\in K_0(A)^+\), 
  there exists \(\delta>0\) such that \(g\leq h\in K_0(A)^+, \abs{g-h}\leq \delta\) implies \(s_g(\tilde{a})-s_h(\tilde{a})\leq \varepsilon/3\). 
  We have
  \begin{align*}
    s_g(a)-s_h(a)&=s_g(a)-s_g(\tilde{a})+s_g(\tilde{a})-s_h(\tilde{a})+s_h(\tilde{a})-s_h(a)\\
    &\leq\norm{s(a)-s(\tilde{a})}+\abs{s_g(\tilde{a})-s_h(\tilde{a})}+\norm{s(\tilde{a})-s(a)}\\
    &<\varepsilon/3+\varepsilon/3+\varepsilon/3=\varepsilon,
  \end{align*}
  for any \(g\leq h\in K_0(A)^+, \abs{g-h}\leq \delta\). 
  Thus \(s(a)\) is right-continuous for all \(a\in A\).
  
  Since \(A\) has real rank zero, 
  there exists a sequence \(\{a_n\}\) of positive elements in \(A\) with finite spectrum 
  such that \(\{a_n\}\) uniformly converges to \(a\).
  By Proposition~\ref{SVF:Basic_Property:1}~(4),
  \(\{s(a_n)\}\) uniformly converges to \(s(a)\).
  By Proposition~\ref{SVF:Example:positive_elements_with_finite_spectrum},
  \(DP(\overline{s(a_n)})\subset K_0(A)^+\).
  Thus by Lemma~\ref{uniform_limit:RC:DP}, \(DP(\overline{s(a)})\subset K_0(A)^+\) for all \(a\in A\).
  Since \(A\) has real rank zero, by Proposition~\ref{SVF:Property:vanishes_at_infty}, \(s(a)\) vanishes at infinity.
  This completes the proof.
\end{proof}

\begin{example}
  The previous theorem holds for (the stabilization of) simple unital C\(^*\)-algebra \(A\) with cancellation such that \(K_0(A)\) is totally ordered.
\end{example}

\subsection{Singular value functions for C\(^*\)-algebras with tracial states}

In this subsection, we equip certain \(K_0\)-groups with topologies arising from real vector spaces and investigate whether singular value functions are lower semicontinuous.

\subsubsection{Ordered groups}

We refer the reader to \cite[pp.~23--24]{Effros:1981:Dimensions_and_C*-algebras}.

Let \((G,G^+,u)\) be an ordered group with a fixed order unit,
and let \(S(G)\) denote its state space.
We denote by \(C_\mathbb{R}(S(G))\) the Banach space of continuous functions \(g\colon S(G)\to\mathbb{R}\), equipped with the sup norm.
We let \(\mathrm{Aff}(S(G))\) denote the set of all affine functions in \(C_\mathbb{R}(S(G))\).
There exists a positive homomorphism
\begin{equation}
  \label{GtoAffSG}
  \rho\colon G\longrightarrow \mathrm{Aff}(S(G)),\qquad g\longmapsto \hat{g},
\end{equation}
where \(\hat{g}\) is the continuous affine function on \(S(G)\) defined by 
\begin{equation*}
  \hat{g}\colon S(G)\longrightarrow \mathbb{R},\qquad f\longmapsto f(g).
\end{equation*}

An element \(g\in G\) is called \textit{infinitesimal} if
\(-m u\leq ng\leq m u\) for all \(m,n\in\mathbb{N}\) with \(m,n>0\).
We denote by \(\mathrm{Inf}(G)\) the set of all infinitesimal elements of \(G\).
The following proposition is due to \cite{Goodearl.Handelman:1976}.

\begin{proposition}
  \label{GtoAffSG:simple}
  Let \((G,G^+,u)\) be a simple weakly unperforated ordered group with a fixed order unit.
  \begin{enumerate}
    \item \(G^+=\{g\in G\mid \hat{g}(f)>0\text{ for all } f\in S(G)\}\cup\{0\}\).
    \item \(\mathrm{Inf}(G)=\ker\rho\).
  \end{enumerate}
\end{proposition}

By Proposition~\ref{GtoAffSG:simple}~{(2)},
if \(G\) is simple and weakly unperforated,
then the kernel of \(\rho\) coincides with \(\mathrm{Inf}(G)\). 
In particular, the induced map \(G/\mathrm{Inf}(G)\to\rho(G)\subset\mathrm{Aff}(S(G))\) is an order isomorphism.

We equip \(\mathrm{Inf}(G)\) with discrete topology, and \(\mathrm{Aff}(S(G))\subset C_\mathbb{R}(S(G))\) with the topology induced by sup norm.

Let \((G,G^+,u)\) be a simple weakly unperforated ordered group with a fixed order unit.
If the short exact sequence
\begin{equation*}
  0\to\mathrm{Inf}(G)\hookrightarrow G\to G/\mathrm{Inf}(G)\to 0,
\end{equation*}
splits, then we equip \(G=\mathrm{Inf}(G)\oplus G/\mathrm{Inf}(G)\) with the product topology induced by the discrete topology on \(\mathrm{Inf}(G)\) and the sup norm topology on \(\mathrm{Aff}(S(G))\). 

\subsubsection{Topologies on \(K_0\)-groups}

The pre-ordered \(K_0\)-groups often behave like ordered groups with order units. 
For example, 
if \(A\) is a stably finite, stably unital C\(^*\)-algebra such that \(M_\infty(A)\) has a full projection \(p\), 
then \((K_0(A), K_0(A)^+, [p]_0)\) is an ordered group with an order unit.
By \cite{Blackadar.Rordam:1992:states_and_quasitraces,Haagerup:2014:Quasitraces_on_exact_C*-algebras_are_traces},
if \(A\) is an exact C\(^*\)-algebra of real rank zero and \(K_0(A)\) is an ordered group with order units,
then \(T(A)\) can be canonically identified with \(S(K_0(A))\).

Applying these facts,
we have the following definition.

\begin{definition}
  \label{K_0-group:Topologies}
  Let \(A\) be an exact
  C\(^*\)-algebra of real rank zero.
  Suppose that \(K_0(A)\) is a simple and weakly unperforated ordered group with an order unit.
  If the short exact sequence
  \begin{equation*}
    0\to\mathrm{Inf}(K_0(A))\hookrightarrow K_0(A)\to K_0(A)/\mathrm{Inf}(K_0(A))\to 0,
  \end{equation*}
  splits, then we equip \(K_0(A)\) with the product topology induced by the discrete topology on \(\mathrm{Inf}(K_0(A))\) and the sup norm topology on \(\mathrm{Aff}(T(A))\).
\end{definition}

\begin{example}
  Let \(A\) be an exact, simple and unital C\(^*\)-algebra of real rank zero with a unique tracial state \(\tau\) (thus \(A\) is stably finite).
  If \(K_0(A)\) is weakly unperforated and 
  \begin{equation*}
    \mathrm{Inf}(K_0(A))=\ker(\rho)=\ker(\tau_*)=\{0\},
  \end{equation*}
  then we can regard \(K_0(A)\subset\mathbb{R}\).
  In this case, if \(A\) has cancellation, then singular value functions on \(K_0(A)^+\) are right-continuous (by Theorem~\ref{SVF:form:simple_and_totally_ordered}).
  It is similar to singular value functions for von Neumann algebras.
\end{example}

\begin{example}
  By the theorem of Effros-Handelman-Shen \cite{Effros.Handelman.Shen:1980}, 
  there exist AF algebras with the following ordered \(K_0\)-groups.
  \begin{enumerate}
    \item \(\left(\mathbb{Q}\oplus\mathbb{Z}, \{(u,v)\in\mathbb{Q}\oplus\mathbb{Z}\mid u>0\}\cup\{(0,0)\}\right)\),
    \item \(\left(\mathbb{Q}^k,\{(x_0,\dots,x_{k-1})\in\mathbb{Q}^k\mid x_i>0 \text{ for all }0\leq i\leq k-1\}\cup\{0\}\right)\),
    \item \(\left(\mathbb{Q}^k\oplus\mathbb{Z},\{(u,v)\in\mathbb{Q}^k\oplus\mathbb{Z}\mid u>0\}\cup\{(0,0)\}\right)\).
  \end{enumerate}
  If these C\(^*\)-algebras are equipped with topologies defined above, then the singular value functions are not necessarily lower semicontinuous.
  
  Let \(A\) be an AF algebra such that \((K_0(A),K_0(A)^+)\) is case (1),
  and let \(p\) be a nonzero projection in \(A\) with \([p]_0=(u,v)\).
  For \((w,z)\in K_0(A)^+\), we have
  \begin{align*}
    s_{(w,z)}(p)&=
    \begin{cases}
      1 & \text{if \((u,v)\not\leq (w,z)\),}\\
      0 & \text{if \((u,v)\leq (w,z)\),}
    \end{cases}\\
    &=
    \begin{cases}
      1 & \text{if \(u>w \text{, or } u=w \text{ and } v\neq z\),}\\
      0 & \text{if \(u<w \text{, or } u=w \text{ and } v=z\),}
    \end{cases}
  \end{align*}
  by Lemma~\ref{SVF:Example:projection}.
  Fix \(z\neq v\). 
  Then \((u+1/n,z)\) converges to \((u,z)\)
  in the product topology, while
  \begin{equation*}
    s_{(u+1/n,z)}(p)=0,\qquad s_{(u,z)}(p)=1.
  \end{equation*}
  for all \(n\in\mathbb{N}_{\geq 1}\).
  Thus \(s(p)\) is not lower semicontinuous.
\end{example}

\section*{Acknowledgments}

The author would like to thank his supervisor, 
Norio Nawata, 
for introducing him to this problem and for many helpful discussions.
The author used ChatGPT for language editing. 

\bibliography{Bib} 
\bibliographystyle{amsalpha} 

\end{document}